\documentclass[reqno,a4paper,11pt]{amsart}

\usepackage{amsmath, graphicx}
\usepackage{amstext,amsfonts,amsbsy,eucal,amssymb}

\numberwithin{equation}{section}

\newtheorem{theorem}{Theorem}[section]
\newtheorem{lemma}[theorem]{Lemma}
\newtheorem{proposition}[theorem]{Proposition}
\newtheorem{corollary}[theorem]{Corollary}

\theoremstyle{definition}

\newtheorem{example}{Example}[section]
\newtheorem{algorithm}{Algorithm}[section]

\theoremstyle{remark}
\newtheorem{remark}{Remark}[section]

\newcommand{\R}{\operatorname{\mathbb{R}}}
\newcommand{\Q}{\operatorname{\mathbb{Q}}}

\newcommand{\C}{\operatorname{\mathbb{C}}}

\newcommand{\I}{\operatorname{\sqrt{-1}}}
\newcommand{\cI}{\operatorname{{\mathcal I}}}
\newcommand{\cV}{\operatorname{{\mathcal V}}}
\newcommand{\rank}{\mathop\mathrm{rank}}
\newcommand{\Rep}{\mathop\mathrm{Re}}
\newcommand{\Imp}{\mathop\mathrm{Im}}
\newcommand{\Proj}{\mathop\mathrm{Proj}}

\def\<{\operatorname{\langle}}
\def\>{\operatorname{\rangle}}

\def\ol{\overline}

\newcommand{\N}{\operatorname{\mathbb{N}}}
\newcommand{\cP}{\mathcal{P}}
\newcommand{\cD}{\mathcal{D}}
\renewcommand{\phi}{\varphi}
\newcommand{\bX}{\bar{X}}
\newcommand{\bp}{\bar{p}}

\newcommand{\bphi}{\bar{\phi}}

\begin{document}

\baselineskip 16pt

\title{
Real ideal and the duality of semidefinite programming for 
polynomial optimization}

\author{Yoshiyuki Sekiguchi}
\address[Y. Sekiguchi]{Faculty of Marine Technology, 
\newline\hphantom{iii}Tokyo University of Marine Science and Technology,
\newline\hphantom{iii}  2-1-6 Etchu-jima, Koto-ku, Tokyo, 
\newline\hphantom{iii}135-8533, Japan}
\email{yoshi-s@kaiyodai.ac.jp}

\author{Tomoyuki Takenawa}

\address[T. Takenawa]{Faculty of Marine Technology, 
\newline\hphantom{iii}Tokyo University of Marine Science and Technology,
\newline\hphantom{iii}  2-1-6 Etchu-jima, Koto-ku, Tokyo, 
\newline\hphantom{iii}135-8533, Japan
}

\email{takenawa@kaiyodai.ac.jp}

\author{Hayato Waki}

\address[H. Waki]{Department of Computer Science,
\newline\hphantom{iii}The University of Electro-Communications,
\newline\hphantom{iii}1-5-1 Chofugaoka, Chofu-shi, Tokyo, 
\newline\hphantom{iii}182-8585 Japan}
\email{hayato.waki@jsb.cs.uec.ac.jp}



\maketitle

\begin{abstract}
We study 
the ideal generated by polynomials vanishing
on a semialgebraic set and
propose an
algorithm to calculate the generators, which is based on 
some techniques of the cylindrical algebraic decomposition.
By applying these,  
polynomial optimization problems with polynomial equality 
constraints can be modified equivalently
so that the associated semidefinite programming relaxation problems 
have no duality gap. Elementary proofs for some 
criteria on reality of ideals are also given.
\end{abstract}

\section{Introduction}
Polynomial Optimization Problem (POP) is a problem for minimizing a
polynomial objective function over a basic closed 
semialgebraic set defined by 
polynomial inequalities and equalities: 
\begin{align}
 \text{minimize } & f(x)\nonumber\\
 \text{subject to } & g_i(x)\geq0,\ i=1,\ldots,\ell;\label{pop}\\
 & h_j(x)=0,\ j=1,\ldots,m,\nonumber
\end{align}
where and $f,g_i,h_j$ are real polynomial functions of $x\in \R^n$.
POP represents various kinds of optimization problems
and can be solved efficiently under moderate assumptions by SemiDefinite Programming (SDP)
relaxations developed by several authors, in particular Lasserre \cite{Lasserre01}
and Parrilo \cite{Parrilo00, Parrilo03}; 
see, for recent developments with equality constraints \cite{Laurent08,
Vo08}
and references therein.

Let $\R[x]$ denote the polynomial ring $\R[x_1,\ldots,x_n]$
and $\R[x]_k$ be the set of polynomials 
with degree up to $k$.
The method constructs sequences
$\{\cP_k\}$ of optimization problems and their dual problems $\{\cD_k\}$ 
from POP (\ref{pop});
\begin{align*}
(\cP_k)\quad \text{minimize } & L(f)\\
 \text{subject to }  & L\colon \R[x]_k\to \R, \text{ linear};\\
 & L(1)=1, L(M_k)\subset [0,\infty)\\
(\cD_k)\quad \text{maximize } & q\\
 \text{subject to } & f-q\in M_k,\ q\in \R,
\end{align*}
where $k$ is an integer greater than or equal to
$k_0:=\max\{\lceil\frac{\deg g_j}{2}\rceil, \lceil\frac{\deg
h_i}{2}\rceil, \deg(f)\}$, and $M_k$ is defined from the constraint
system of POP (\ref{pop}):
\[
 M_k=\left\{\left.\sum_{i=0}^{\ell}\sigma_i g_i+\sum_{j=1}^mr_jh_j\right|
 \sigma_i\in \sum \R[x]^2,
 r_j\in \R[x], \deg(\sigma_ig_i)\leq k, \deg(r_jh_j)\leq k\right\}.
\]
Here $g_0(x) = 1$ and $\Sigma\R[x]^2$ is the set of the sum of square
polynomials. The union $M$ of all $M_k$ is called the quadratic module
generated by $g_1, \ldots, g_{\ell}$ and $h_1, \ldots, h_m$. 
Let $f_k^*$
and $q_k^*$ be the optimal values of $\cP_k$ and $\cD_k$,
respectively. Lasserre \cite{Lasserre01} formulated them as SDP problems, 
and showed that the sequences $\{f_k^*\}_{k\geq k_0}$
and $\{q_k^*\}_{k\geq k_0}$ converge to the optimal value of the given POP
under moderate assumptions.
In addition, he showed that if the feasible region has nonempty interior (no
equality constraints), the equality $f_k^*=q_k^*$ holds for any $k\ge
k_0$; SDP has no duality gap.
On the other hand, Marshall \cite{Marshall03} focused on the quadratic
module $M$ from POP (\ref{pop}) instead of the basic semialgebraic set,
and then proved that
$f_k^*=q_k^*$  if
$\cI(K)\subset M$ holds,where $K=\{x\in \R^n\mid g_i(x)\geq0,
i=1,\ldots,\ell, h_j(x)=0,j=1,\ldots,m\}$; the feasible region of POP
(\ref{pop}) and $\cI(K)=\{p\in \R[x]\mid p(x)=0, \forall x\in K\}$; 
the vanishing ideal of $K$. 
However, it is difficult to check this
assumption for a given POP.

In this paper, we study the method through investigations on vanishing
ideals of general semialgebraic sets in $\R^n$.
Let $\cV(I)=\{x\in \R^n\mid p(x)=0, \forall p\in I\}$ for an ideal $I\subset\R[x]$.
When we deal with a polynomial ring over $\C$,
the Hilbert's Nullstellensatz describes a relationship between varieties and ideals.
On the other hand, for an ideal $I$ in $\R[x]$,
the Real Nullstellensatz says that $\cI(\cV(I))=I$ if $I$ is real; 
see Section 3 for the details.
We give elementary proofs for some criteria on reality of ideals 
(Theorem \ref{realcond} and Theorem \ref{condtopdim}).
Then we discuss equivalent conditions for the equality $\cI(S\cap \cV(I))=I$, 
where $S$ is a semialgebraic set (Theorem \ref{main_condition} and \ref{condition}).
Both conditions $\cI(\cV(I))=I$ and $\cI(S\cap \cV(I))=I$ are verifiable
and closely
related 
to duality of SDP (Proposition \ref{duality}).
In addition, we propose an
algorithm to calculate generators of $\cI(K)$, using
some techniques of the Cylindrical Algebraic Decomposition
(CAD) after Collins \cite{BPR06, Mishra93}.
Applying these, one can equivalently modify any POP 
so that associated semidefinite programming relaxation problems 
have no duality gap. 
No duality gap of SDP is an important property theoretically and practically.
For example, it is one of fundamental conditions for convergence
of interior point methods, or it is used to confirm optimality of 
a solution.

This paper is organized as follows: In Section 2, we present a
relationship between the vanishing ideal of $K$ and 
the duality of SDP for POP (\ref{pop}).
In Section 3, elementary proofs are given for some criteria on reality of
ideals and equivalent conditions are obtained
for the equality $\cI(S\cap \cV(I))=I$.
Algorithms for deciding reality of ideals and for calculating 
generators of $\cI(S\cap \cV(I))$ are given
in Section 4.

\section{Duality}
We discuss the duality of SDP relaxation problems for POP (\ref{pop}).
We rewrite POP (\ref{pop}) as follows.
\begin{equation}
  \text{minimize } f(x)
 \text{ subject to } x \in K:=S\cap \cV(I),\label{pop2}
\end{equation}
where $S = \{x\in\R^n\mid g_j(x)\ge 0, j=1,\ldots, \ell\}$ and  
 $I=\langle h_1,\ldots,h_m\rangle$ is the ideal generated by $h_1,\ldots,h_m$ in $\R[x]$.
The following proposition is a sufficient condition for the duality. 
\begin{proposition}
\label{duality}
Suppose that $\cI(K) = I$. 
Then $f_k^*=q_k^*$.
Moreover if there is a feasible point in $(\cD_k)$,
it has an optimal solution.
\end{proposition}
\begin{proof}
If $M_k$ is closed in the Euclidean topology, the similar arguments in
 \cite[Corollary 21]{Sch05}
ensure $f_k^*=q^*_k$ and existence of an optimal solution. Closedness of 
$M_k$ is shown in 
the later paragraphs of this section (Theorem~\ref{closed}).
\end{proof}
It should be noted that Marshall \cite{Marshall03} 
has shown a similar result.
Under the assumption $\cI(K)\subset M$, he showed
closedness of $M_k/(\cI(K)\cap \R[x]_k)$.
Although our assumption is slightly stronger than his, 
it can be verified directly as given below. 
The following theorem gives
one of verifiable conditions for our assumption. 
The proof is given in Section 3
and a decision algorithm for the condition (2) is given in Section 4.

\begin{theorem}
\label{main_condition}
Let $S$ be a semialgebraic set in $\R^n$,
$I$ be an ideal in $\R[x]$ and let $I=I_1\cap \cdots \cap I_k$ be 
the prime decomposition of $I$. The following conditions are equivalent.
\begin{enumerate}
 \item $\cI(S\cap \cV(I))=I$;
 \item For any $t$ ($1\leq t\leq k$), $\dim^{{\rm top}} (S\cap \cV(I_t)) = \dim I_t$
 holds, where $\dim^{{\rm top}}(S\cap \cV(I_t))$ is the topological
 dimension of $S\cap \cV(I_t)$.    
\end{enumerate}
\end{theorem}
Here the dimension $\dim I$ of an ideal $I$ is defined in Section 3
($\dim \R[x] =\dim^{{\rm top}} \emptyset =-1$).
Theorem~\ref{main_condition} becomes simpler if $S^\circ\cap
\cV(I)$ is nonempty, where $S^\circ$ is the interior of $S$.
The following is a corollary of Theorem 3.13 in Section 3.
\begin{corollary}
\label{duality_of_prime}
 Suppose that $I=\langle h_1,\ldots,h_m\rangle$ is prime.
 If there exists a feasible point $x_0$ for POP $(\ref{pop})$ such that
 $x_0\in S^{\circ}$ and 
the rank of the Jacobian matrix 
$\frac{\partial (h_1,\ldots,h_m)}{\partial (x_1,\ldots,x_n)}(x_0)$ is equal to $n-\dim I$, then $f_k^*=q_k^*$.
\end{corollary}

\begin{example}
We consider the following POP.
\begin{align}
 \text{minimize } & f(x)\nonumber\\
 \text{subject to } & g_i(x)\geq0,\ i=1,\ldots,{\ell};\label{pop_ex}\\
 & h_j(x):=a_j^Tx -b_j=0,\ j=1,\ldots,m,\nonumber
\end{align}
where $f,g_i\in \R[x]$, $a_j\in\R^n$ and $b_j\in\R$. We assume that
$S^{\circ}\cap \cV(I)$ is nonempty. Then it follows from Corollary
~\ref{duality_of_prime} 
 that SDP relaxation problems for POP (\ref{pop_ex}) have no duality gap. Indeed, it
 is clear that $\langle h_1,\ldots, h_m \rangle$ is prime. In addition,
 $\rank\frac{\partial (h_1,\ldots,h_m)}{\partial (x_1,\ldots,x_n)}
=n-\dim\ker[a_1,\cdots,a_m]=n-\dim I$.
\end{example}

To show closedness of $M_k$, we start with the following technical lemma.
For $d\in \N$, let $\Lambda(d)=\{\alpha\in
\N^n\mid|\alpha|= 
\alpha_1+\cdots+\alpha_n\leq d\}$.  For $k\geq k_0$, let
$d_i=\max\{d\in \N\mid 2d+\deg g_i\leq k\}$, 
$e_i=k-\deg h_i$, 
and $X=\prod_{i=0}^\ell\R[x]_{d_i}^{\Lambda(d_i)}\times \prod_{j=1}^m
 \R[x]_{e_j}$. We define a mapping $\phi\colon X\to M_k$ by
\[
 q=\left((q_{0\alpha})_{\alpha\in \Lambda(d_0)},\ldots,
 (q_{\ell\alpha})_{\alpha\in \Lambda(d_\ell)}, r_1,\ldots,r_m\right)
\mapsto \sum_i\sum_\alpha(q_{i\alpha})^2g_i+\sum_jr_jh_j.
\]
It is known that $\phi$ is surjective
(the Gram matrix description of sums of squares); 
see for instance \cite{Sch05}.

\begin{lemma}
\label{lemma_assump}
Under the same assumption of Proposition~\ref{duality},
$\phi(q)\in I$ if 
and only if $q_{i\alpha}$ belongs to the quotient of ideals $(I\colon
 \langle g_i\rangle)=\{s\in \R[x]\mid sg_i\in I\}$ for all $\alpha\in
 \Lambda(d_i)$ and $i=0,1,\ldots,\ell$. 
\end{lemma}
\begin{proof}
 Suppose $\phi(q)\in I$. Then 
 $\sum_i\sum_\alpha(q_{i\alpha})^2g_i$ belongs to $I$ 
and hence vanishes on $K$.
Since $g_i(x)\geq0$, we have each $q_{i\alpha}^2g_i=0$ on $K$
and thus $q_{i\alpha}g_i=0$. Therefore
 the assumption implies $q_{i\alpha}g_i\in I$. 
The converse is obvious.
\end{proof}

Let $\R[x]_k$ be endowed with Euclidean topology.
The following theorem is a slight modification of 
Theorem 3.1 of Marshall
\cite{Marshall03}.

\begin{theorem}\label{closed}
  Under the same assumption of Proposition~\ref{duality}, $M_k$ is
 closed.
\end{theorem}
\begin{proof}
 Let $J_{d_i}=\{s\in \R[x]_{d_i}\mid sg_i\in I\}$ and
 $I_{e_j}=I\cap \R[x]_{e_j}$.
Then $J_{d_i}$ and $I_{e_j}$ are closed subspaces 
of vector spaces $\R[x]_{d_i}$ and $\R[x]_{e_j}$ respectively. 
We define $\bX=\prod_{i,\alpha}\bX_{i\alpha}\times \prod_j\bX_j$, 
where $\bX_{i\alpha}=\R[x]_{d_i}/J_{d_i}$ for $\alpha\in \Lambda(d_i)$, 
$i=0,1,\ldots,\ell$  and $\bX_j=\R[x]_{e_j}/I_{e_j}$ for $j=1,\ldots,m$.
Then $\bX$ is a normed space.

Let $\bphi\colon \bX\to M_k/I_k$ be the induced mapping by $\phi$.
Then we have $\bphi$ is surjective, $\bphi(\lambda q)=\lambda^2\bphi(q)$
 for $\lambda\in \R$, 
and $\bphi^{-1}(0)=\{0\}$ by Lemma~\ref{lemma_assump}.
Hence 
we have $M_k/I_k=\{\lambda v\mid \lambda\in [0,\infty),v\in V\}$, where
$V$ is the image of the unit sphere in $\bX$ under $\bphi$.
In addition, $V$ is compact and does not contain zero element.
Now we suppose $p_s\in M_k$ and $p_s\to p$ in $\R[x]_k$. 
Let $\bp_s$ and $\bp$ be the cosets of $p_s$ and $p$ respectively.
Then there exist $\lambda_s\geq 0$ and $v_s\in V$ 
such that $\bp_s=\lambda_sv_s$.
By compactness of $V$, we may assume $v_s$ converges to some element
 $v\in V$. Then the limit of $\lambda_s$ exists, since
 $\lambda_s=\|\lambda_sv_s\|/\|v_s\|$ converges to $\|\bp\|/\|v\|$. 
Therefore we have $\bp=\lim_s \lambda_sv_s=\|\bp\|/\|v\|v\in M_k/I_k$
and hence $p\in M_k$.

\end{proof}

\section{Real ideals and semialgebraic sets}

The first main result of this section is 
Theorem~\ref{realcond}, which is called the simple point criterion
for reality of ideals
and has already been proved\footnote{
This assertion is also written 
as Proposition~3.3.16 in \cite{BCR98}, but the proof has
a gap.} by M. Marshall \cite{Marshall08}. 
In this section we give an elementary proof of this theorem
and introduce another equivalent condition (Theorem~\ref{condtopdim}). 
Using these, we also obtain
some conditions equivalent to $\cI(\cV(I)\cap S^\circ)=I$
or to $\cI(S\cap \cV(I))=I$ for a semialgebraic set $S$
(Theorem~\ref{condition} and Theorem~\ref{main_condition}).

Let $x=(x_1,\dots,x_n)$ and $k=\R$ or $\C$. For an ideal 
$I$ of $k[x]$, $\cV(I)$ denotes the 
set of all zero points of $I$ in $k^n$. 
For a subset $V$ of $k^n$, $\cI(V)$ denotes the set of all polynomials 
witch vanish on $V$. 
A semialgebraic set in $\R^n$ is a subset of the form
\begin{align}\label{semialgebraic}
\bigcup_{i=1}^s \bigcap_{j=1}^{r_i} \{x\in \R^n ~;~ g_{i,j} *_{i,j} 0\},
\end{align}
where $f_{i,j} \in \R[x]$ and $*_{i,j}$ is either $>$, $\geq$ or $=$.

For an ideal $I$ of $k[x]$, the dimension of $I$, $\dim I$, 
is the transcendence degree of $I$.
If $I$ is $k[x]$ itself, $\dim k[x]$ is defined as $-1$.
If $I$ is prime, $\dim I$ coincides with the depth of $I$
(Krull dimension of $k[x]/I$). 
For an ideal $I$ of
$k[x]$ with primary decomposition $I=I_1\cap \dots\cap I_k$, 
$\dim I$ is $\max_{t=1,\dots,k} \dim \sqrt{I_t}$.
The dimension does not depend on extensions of the coefficient field.   
See \cite{Marshall08,Weil62} for more details. Finally, the rank
of an ideal $I=\<f_1,\dots,f_m\>$ is the maximal rank of the Jacobian
$\frac{\partial (f_1,\ldots,f_m)}{\partial (x_1,\ldots,x_n)}$
in $\cV(I)$. The topological dimension $\dim^{\mathrm{top}}\cV(I)$ of $\cV(I)$
is the maximal dimension as manifolds.

For a polynomial $f(x)=\sum_a f_a x^a \in \C[x]$, 
$\ol{f}$ or $\ol{f}(x)$ denotes $\sum_a \ol{f_a} x^a$. 
Note that $\ol{f}(x) \neq \ol{f(x)}$ for $x \in \C^n$.

\begin{remark}
The results and the proofs of this section are still valid if $\R$ and $\C$ are
replaced by real closed field and its extension with $\sqrt{-1}$ respectively.
\end{remark}

An ideal $I$ in $\R[x]$ is called real if the following equivalent conditions
are satisfied.
\begin{enumerate}
\item $\cI(\cV(I))=I$;
\item For any integer $m$ and any $f_1,\dots,f_m \in \R[x]$, 
 the equation $\sum_{i=1}^m f_i^2\in I$ implies $f_i\in I$ for all $i$.
\end{enumerate}

We give the equivalent condition for an ideal to be real. 
The following proposition is well known (see Lemma~2.5 of \cite{Efroymson74}
for example).

\begin{proposition}\label{decomposition}
Let $I=I_1\cap\cdots \cap I_k$ be a primary decomposition of 
an ideal $I\subset \R[x]$. Then $I$ is real if and only if
each $I_t$ is prime and real. 
\end{proposition}

For a while, we assume that ideals are prime, and investigate 
the reality.

\begin{lemma}
[Theorem~1.11 of \cite{DF70}]\label{Iprime}
If a prime ideal $I$ in $\R[x]$ is real,
then $I'=\C[x] I:=\{ph\mid p\in
 \C[x],h\in I\}$ is prime in $\C[x]$.
\end{lemma}

\begin{proof}
By assumption, $I \subsetneq \R[x]$. 
Suppose $I=\<f_1,\dots,f_k\>\subset\R[x]$ to be prime in $\R[x]$ and $I'$ not to be prime in $\C[x]$. 
There exist $a,b,c,d \in \R[x]$ such that
$a+b\I, c+d\I \not\in I'$ and 
\begin{align*} 
(a+b \I)(c+d\I)&=\sum (u_i+v_i \I)f_i \quad (u_i,v_i \in \R[x]),
\end{align*}
and thus,
\begin{align*}
 ac-bd&=\sum u_i f_i \mbox{ and } ad+bc= \sum v_i f_i\in I.
\end{align*}
These equations yield that $c(a^2+b^2), d(a^2+b^2)$ are contained in
 $I$. Since
$c+d\I \not\in I'$, both of $c$ and $d$ are not in $I$. Hence $a^2+b^2$ belongs
to $I$. By reality of $I$, both $a$ and $b$ belong to $I$. This implies
that $a+b\I$ belongs to $I'$, which is a contradiction.   
\end{proof}

\begin{lemma}[Lemma~3.4 of \cite{DF70}]\label{DF3.4}
If a prime ideal $I$ in $\R[x]$ is real, 
then $\rank I +\dim I=n$.
\end{lemma}

\begin{proof}
Fix the generators $f_1,\dots,f_k$ of $I$. 
From Lemma~\ref{Iprime}, $I'$ is prime in $\C[x]$, and thus
$\rank I' +\dim I'= n$ holds, because $\C$ is algebraically closed.
Suppose $\rank I +\dim I<n$ on $\cV(I)$ 
and let $h_i$'s be the sub-determinants
of Jacobian of size $r=n-\dim I$. From the assumption,
there exists $i$ such that $h_i$ is identically zero on $\cV(I)$ but
is not identically zero on $\cV(I')$.
Such $h_i$ does not belong to $I'$, thus $h_i\not\in I$,
which contradicts reality of $I$.
\end{proof}

The following lemma can be shown similarly to the proof  
of Lemma~\ref{Iprime},

\begin{lemma}
\label{gg*}
If $I$ is a prime ideal in $\R[x]$ 
and $I'=\C[x] I$ is not prime in $\C[x]$.
Then, there exists an irreducible polynomial 
$g(x)$ in $\C[x]\setminus I' $ such that $g(x)\ol{g}(x)\in I$.
\end{lemma}

\begin{proof}
By assumption, $I \subsetneq \R[x]$. 
Suppose $I=\<f_1,\dots,f_k\>\subset\R[x]$ to be prime and $I'$ not to be prime. 
Similar to the proof of Lemma~\ref{Iprime}, there exist $a,b\in \R[x]$
 such that
$a+b\I \not\in I'$ and  $a^2+b^2 \in I$.
Set $g=a+b\I$. Then $g \in \C[x]\setminus I'$ 
and $g\ol{g}=a^2+b^2 \in I$.

If $g$ is factorized as $g=g_1g_2$, then $g_1,g_2 \not\in I'$ and
$g\ol{g}=(g_1\ol{g_1})(g_2 \ol{g_2})\in I$. Since $I$ is prime,
$g_1\ol{g_1}$ or $g_2 \ol{g_2}$ belongs to $I$. 
Reset $g$ as suitable one.
\end{proof}

\begin{example}
(i) The ideal $J=\<(x^2+y^2)z\>$ is decomposed as $J=\<x^2+y^2\>\cap \<z\>$.
Here, $I=\<x^2+y^2\>$ is prime, while $I'$ is decomposed as 
$I'=\<x+y\I\>\cap \<x-y\I\>$, thus $I$ is not real, neither is $J$.\\ 
(ii) The ideal $I=\<x^2+y^2+z^2\>$ is prime and $I'$ is also prime.
However, since the rank of $I$ is zero, $I$ is not real.\\ 
(iii) 
The ideal $I=\<x^2+y^2,z^2+w^2, xz+yw, xw-yz\>$ is prime in $\R[x]$, while
$I'$ is decomposed as $I'=\<x+y\I, z+w\I\>\cap \<x-y\I, z-w\I\>$, thus $I$
is not real.
This example is essential for the proof of the following proposition. 
\end{example}

The following proposition is essentially a special case of Theorem~9.3 
of \cite{Matsumura89}, which is related to going-up theorem.
We give an elementary and constitutive proof. 

\begin{proposition}\label{Ig}
If $I$ is a prime ideal in $\R[x]$ and $I'=\C[x]$ is not prime in $\C[x]$.
Then, there exist irreducible polynomials
$g_1(x),\dots,g_k(x) \in \C[x]\setminus I' $ such that 
$g_1(x)\ol{g_1}(x)$, $\dots$, $g_k(x) \ol{g_k}(x)$ belong to $I$ 
and
$I'=\<I',g_1,\dots,g_k\>\cap \<I',\ol{g_1},\dots,\ol{g_k}\>$ (denoted as
$I_g' \cap I_{\ol{g}}'$)
is the prime decomposition of $I'$.
\end{proposition}

\begin{proof}
From Lemma~\ref{gg*}, there exists an irreducible polynomial
$g_1\in \C[x]\setminus I' $ s.t. $g_1(x)\ol{g_1}(x)\in I$ 
and
$I'=\<I',g_1\>\cap \<I',\ol{g_1}\>$.
If we set $g_1=a_1+b_1\I$, since $I$ is prime, 
we have $a_1,b_1 \in \R[x]\setminus I$. 

We decompose $I'$ inductively with respect to $k$. 

Suppose $I'$ is decomposed as 
\begin{align}\label{Igk} 
I'&= I_g' \cap I_{\ol{g}}'
\end{align}
($I'\subsetneq I_g', I_{\ol{g}}'$) as the assertion.
If we set $g_i=a_i+b_i\I$, since $I$ is prime, $a_i,b_i \in \R[x]\setminus I$. 
If both $I_g'$ and $I_{\ol{g}}'$ are prime, the assertion follows.
Suppose $I_g'$ is not prime 
(which implies $I_{\ol{g}}'$ is not prime either). 
There exist $g_{k+1}, h_{k+1} \in \C[x]\setminus I_g'$
s.t. $g_{k+1}h_{k+1}\in I_g'$. We have
$$g_{k+1}h_{k+1} \ol{g_{k+1}h_{k+1}} \in I_g' \cap I_{\ol{g}}' \cap \R[x] =I,$$
$$g_{k+1}\ol{g_{k+1}} \in \R[x] \mbox{ and }  h_{k+1}\ol{h_{k+1}} \in \R[x].$$
Since $I$ is prime, $g_{k+1}\ol{g_{k+1}}$ or $h_{k+1}\ol{h_{k+1}}$ belongs to 
$I$.
We can assume $g_{k+1}\ol{g_{k+1}}\in I$. 
Similar to the proof of Lemma~\ref{gg*}, 
we can also assume $g_{k+1}$ is irreducible.
Set $g_{k+1}=a_{k+1}+b_{k+1}\I$ 
($a_{k+1},b_{k+1} \in \R[x]\setminus I$). 
Then we have, for $1\leq i\leq k$,
\begin{align*}
a_i^2+b_i^2, a_{k+1}^2+b_{k+1}^2 &\in I,\\
a_i^2a_{k+1}^2-b_i^2b_{k+1}^2, a_i^2b_{k+1}^2-a_{k+1}^2b_i^2 &\in I, \\
a_ia_{k+1}+b_ib_{k+1}\mbox{ or } a_ia_{k+1}-b_ib_{k+1} &\in I,\\
a_ib_{k+1}+a_{k+1}b_i\mbox{ or } a_ib_{k+1}-a_{k+1}b_i & \in I. 
\end{align*}
When $a_ia_{k+1}+b_ib_{k+1}, a_ib_{k+1}+a_{k+1}b_i \in I$ for some $i$, 
we have 
$(a_ia_{k+1}+b_ib_{k+1})^2-(a_ib_{k+1}+a_{k+1}b_i)^2
=(a_i^2-a_{k+1}^2)(b_i^2-b_{k+1}^2)\in I$.
Thus, $(a_i^2-b_i^2)$ or $(a_{k+1}^2-b_{k+1}^2)$ is in $I$.
Since $I$ is prime,
the former implies $a_i, b_i \in I$ and
the latter implies $a_{k+1}, b_{k+1} \in I$,
each of which is a contradiction.   
Similarly, $a_ia_{k+1}-b_ib_{k+1}, a_ib_{k+1}-a_{k+1}b_i \in I$ 
yields a contradiction.

Suppose $a_ia_{k+1}+b_ib_{k+1}, a_ib_{k+1}-a_{k+1}b_i \in I$
for all $1\leq i \leq k$.
$I'$ is decomposed as $I'= J \cap J^* $ where
\begin{align*}
J=&\<I', a_1+b_1\I, \dots, a_k+b_k\I, a_{k+1}+b_{k+1}\I \>, \\
&\cap \<I',  a_1-b_1\I, \dots, a_k-b_k\I, a_{k+1}-b_{k+1}\I\>,\\
J^*=&\<I',  a_1+b_1\I, \dots, a_k+b_k\I, a_{k+1}-b_{k+1}\I \>,\\
 &\cap \<I',  a_1-b_1\I, \dots, a_k-b_k\I, a_{k+1}+b_{k+1}\I\>.
\end{align*}
We show $I'=J$, i.e.
\begin{align*}
(a_i+b_i\I)(a_i-b_i\I),\ (a_i+b_i\I)(a_{k+1}-b_{k+1}\I),\\
(a_{k+1}+b_{k+1}\I)(a_i-b_i\I),\   
(a_{k+1}+b_{k+1}\I)(a_{k+1}-b_{k+1}\I)
\end{align*} belong to $I'$.
This means
\begin{align*}
a_i^2+b_i^2,\ (a_ia_{k+1} +b_ib_{k+1})- (a_i b_{k+1} -a_{k+1}b_i)\I,\\
a_ia_{k+1}+b_ib_{k+1} +(a_ib_{k+1}-a_{k+1}b_i)\I,\ 
a_{k+1}^2+b_{k+1}^2
\end{align*} 
belong to $I'$, which is obvious by the assumption.
Similarly, the assumption that 
$a_ia_{k+1}-b_ib_{k+1}, a_ib_{k+1}+a_{k+1}b_i \in I$ for all 
$1\leq i\leq k$ yields $I'=J^*$.
Exchanging $g_{k+1}$ and $\ol{g_{k+1}}$, we have
$$I'= \<I_g', g_{k+1} \> \cap 
\<I_{\ol{g}}', \ol{g_{k+1}}\>.$$

Since $I' \subsetneq \<I',g_1\>\subsetneq \<I',g_1,g_2\> \subsetneq \dots$ is
an ascending chain, 
this procedure terminates in finite steps.
Finally, we show that if $I'$ is primarily decomposed as \eqref{Igk},
this is the prime decomposition.
Suppose $g^m \in I_g'$ for some $m>1$ and
$g \not\in I_g'$. Then, 
$g^m\ol{g}^m \in I_g' \cap I_{\ol{g}}' \cap \R[x]=I.$
Since $I$ is prime, $g\ol{g}\in I$. By construction, 
we have $g \in I_g'$. 
\end{proof}

\begin{corollary}
Let $I$ be a prime ideal in $\R[x]$, then 
$\dim I'+ \rank I'=n$ holds for $I'=\C[x] I$. Moreover if 
$I'$ is not prime in $\C[x]$, then
$\dim I'=\dim I_g'=\dim I_{\ol{g}}'$ and
$\rank I'=\rank I_g'=\rank I_{\ol{g}}'$ hold.
\end{corollary}

\begin{proof}
If $I'$ is prime, the assertion is obvious. We assume $I'$ not to be prime.
We show the latter assertion, which leads the former immediately.
By the general theory,
$\dim I' =\max \{\dim I_g', \dim I_{\ol{g}}' \}$ holds.
Now, by symmetry, it is clear that $\dim I_g'=\dim I_{\ol{g}}' $.
Next, we show the equation for the rank. 
We can assume $\cV(I') \neq \emptyset$.
By the symmetry of
$$\cV(I')=
\cV(I_g') \cup \cV(I_{\ol{g}}'),$$
it implies $\cV(I_g') \neq \emptyset$.
Set $g_i=a_i+b_i\I$. If $a_i \in I_g'$, then 
$a_i \in I_g'\cap I_{\ol{g}}'\cap \R[x]=I$, which is a contradiction.
Since $I_g' \subsetneq \<I_g',a_i\>$ and $I_g'$ is prime,
we have $\dim I_g' > \dim \<I_g', a_i\>$ for all $i$, and hence   
$$\dim (\cV(I_g') \cap 
(\cV \<a_1\> \cup \dots \cup \cV \<a_k\>))< \dim \cV(I_g').$$ 
Using 
$\partial_{x_j}(a_i^2+b_i^2)|_{g_i=0}=2a_i \partial_{x_j}(a_i+b_i\I)$,
we have
\begin{align*}
\rank I_g' 
&= \max_{\{z|\forall i, a_i \neq 0 \}} \rank I_g'\\
&= \max_{\{z|\forall i, a_i \neq 0\}}  \rank I'  \leq \rank I'.
\end{align*}  
The opposite inequality 
$\rank I' \leq \rank I_g'= \rank I_{\ol{g}}'$
is obvious.
\end{proof}

The following lemma is very elementary, but the authors could not find
it in
any literature.

\begin{lemma}
Let $k\geq 2$ and $I=I_1\cap \dots \cap I_k$ 
the prime decomposition of an ideal $I$ in $\C[x]$.
Then the rank of $I$ on $\cV(I_1) \cap \dots \cap \cV(I_k)$ is 
less than $n-\dim I$.
\end{lemma}

\begin{proof}
Without loss of generality we can assume $\dim I_1=\dim I=:d$.
Suppose $\rank_x I \geq n-d$ for some $x \in \cap_{i=1}^k \cV(I_i)$.
Then, there exist $f_1,\dots,f_{n-d} \in I$ such that 
$\rank_x \<f_1,\dots,f_{n-d}\>=n-d$ and
from the implicit function theorem, the topological dimension 
of $\cV \<f_1,\dots,f_{n-d}\>$ in a neighborhood 
$U_x$ of $x$ is $d$. 
Since $I_1$ is prime in $\C[x]$, 
$\dim^{{\rm top}} \cV(I_1)$ is also $d$, 
which implies that a polynomial $g \in I_1$
is identically zero on
$\cV \<f_1,\dots,f_{n-d}\> \cap U_x$, unless otherwise 
$\dim^{{\rm top}} \cV(I_1)<d$.  
From the inclusion $\cV(I_i) \cap U_x \subset 
\cV \<f_1,\dots,f_{n-d}\> \cap U_x$,
$g$ is also identically zero on 
$\cV(I_i) \cap U_x$ for all $1\leq i\leq k$.
Since $I_i$ is prime, $g=0$ on $\cV (I_i)$,
which implies $g \in I_i$. Thus, $g \in \cap_{i=1}^k I_i=I$,
which implies $I=I_1$; contradicts $k\geq 2$.
\end{proof}

\begin{proposition}\label{notprime}
If $I$ is a prime ideal in $\R[x]$ and $I'=\C[x] I$ is not prime in $\C[x]$.
Then $\dim I +\rank I < n.$
\end{proposition}

\begin{proof}
By Proposition~\ref{Ig}, 
$I'$ is decomposed as \eqref{Igk}.
By the above lemma, it is enough to show 
$\cV(I) \subset \cV(I_g') \cap \cV(I_{\ol{g}}')$.
Since $g_i \ol{g_i} \in I$ for all $i$, $x \in \cV(I)\subset \R^n$ 
implies $g_i(x)=\ol{g_i}(x)=0$ and hence
$x \in \cV(I_g') \cap \cV(I_{\ol{g}}')$.
\end{proof}

\begin{lemma}[Lemma~3.9 of \cite{DF70}]\label{DF3.9}
For a prime ideal $I$ in $k[x]$, $I$ is real if and only if
$\rank \cI(\cV(I))=\rank I$.
\end{lemma}

\begin{lemma}\label{topdim}
Assume that $I$ is prime in $\R[x]$ and $I'$ is prime in $\C[x]$, then
$\dim I+\rank I=n$ implies $\dim^{{\rm top}} \cV(I) = \dim I$.  
\end{lemma}

\begin{proof}
Set $d=\dim I$, then $\dim^{{\rm top}} \cV(I')$ is $d$.
Let $x^0$ be a point in $\cV(I)$ such that $\rank_{x^0} I= n-d.$   
Let $f_1,\dots,f_{n-d}\in I$ satisfy $\rank \<f_1,\dots,f_{n-d}\>=n-d$.
By a suitable reordering of the variables, the equations $f_i=0$ can be solved 
for the first $n-r$ variables as functions of the last $d$ variables in a
neighborhood of $x^0$. Let $u_1,\dots,u_{n-r}$ be such solution functions. 
We write $\tilde{x}=(x_{n-d+1},\dots,x_n)$, $u=(u_1,\dots,u_{n-d})$
and $\tilde{f}(\tilde{x})=f(u(\tilde{x}),\tilde{x})$ for $f\in \R[x]$.
Then $$x=(u(\tilde{x}),\tilde{x})$$
holds for all $\tilde{x}$ in a neighborhood of $\tilde{x}^0$.
Let $f$ belong to $I'$, 
then we have $\tilde{f}(\tilde{x})=0$ for all $\tilde{x}$,
unless otherwise
the dimension of the manifold $\cV(I')$ is less than $d$, which
is a contradiction.
Hence $\tilde{x}$ is a coordinate of the manifold 
$\cV(I')$ in a neighborhood of $x^0$. Moreover, since 
$f_1,\dots,f_{n-d}\in \R[x]$, $\tilde{x}$ is also a coordinate of 
the real manifold $\cV(I)$.
\end{proof}

\begin{theorem}
\label{realcond}
Let $I$ be a prime ideal in $\R[x]$\\
i) If $I$ is real, then $I'$ is prime and $\dim I+\rank I=n$. \\
ii) If $\dim I+\rank I=n$, then $I$ is real.
\end{theorem}

\begin{proof}\ \\
i)\ 
It follows from Lemma~\ref{Iprime} and Lemma~\ref{DF3.4}.\\
ii)\  From Proposition~\ref{notprime}, $I'$ is prime. 
We set $f_1,\dots,f_{n-d}$ and notations 
as in the proof of Lemma~\ref{topdim}.
Suppose $f(x)=0$ on $\cV \<f_1,\dots,f_{n-d}\>$.
The derivatives of $\tilde{f}(\tilde{x})$ are zero, so 
\begin{align*}
&0= \frac{\partial \tilde{f}}{\partial \tilde{x}_j}
= \sum_{m=1}^{n-d} \frac{\partial f}{\partial x_m}
\frac{\partial u_m}{\partial \tilde{x}_j}
+ \sum_{m=n-d+1}^{n}
\frac{\partial f}{\partial x_m}
\frac{\partial x_m}{\partial \tilde{x}_j},\\
\intertext{thus we have}
&\frac{\partial f}{\partial x_j}
=-\sum_{m=1}^{n-d} \frac{\partial f}{\partial x_m}
\frac{\partial u_m}{\partial \tilde{x}_j}\quad
(n-d+1\leq j\leq n),
\end{align*} 
and hence
\begin{align*}
\left(\frac{\partial f}{\partial x_1},\dots,
\frac{\partial f}{\partial x_n}\right)=&
\left(\frac{\partial f}{\partial x_1},\dots,
\frac{\partial f}{\partial x_{n-d}}\right) 
\left[\begin{array}{ccccccc}
1 &0 &\dots& 0& \frac{-\partial u_1}{\partial x_{x-d+1}} 
&\dots& \frac{-\partial u_1}{\partial x_n}\\
0& 1 &\dots& 0& \frac{-\partial u_2}{\partial x_{x-d+1}} 
&\dots& \frac{-\partial u_2}{\partial x_n}\\
\vdots&\vdots&&\vdots&\vdots&&\vdots\\
0& 0 &\dots& 1& \frac{-\partial u_{n-d}}{\partial x_{x-d+1}} 
&\dots& \frac{-\partial u_{n-d}}{\partial x_n}
\end{array}\right],
\end{align*}
which implies $\rank \cI(\cV(I)) = n-d$. From Lemma~\ref{DF3.9},
$I$ is real.
\end{proof}

\begin{theorem}
\label{condtopdim}
A prime ideal $I$ in $\R[x]$ is real if and only if $
\dim^{{\rm top}} \cV(I)=\dim I$.
\end{theorem}

\begin{proof}
The "Only if" part follows from Theorem~\ref{realcond} i) 
and Lemma~\ref{topdim}.

We show the "if" part.
Suppose $\dim^{{\rm top}} \cV(I)=\dim I$.  
If $I'=\C[x]$ is not prime, then 
$\cV(I)$ is included in $\cV(I_g') \cap \cV(I_{\ol{g}}')$ 
as in the proof of Proposition~\ref{notprime}, and hence we have 
$\dim^{{\rm top}} \cV(I)< \dim I$. 
Thus $I'$ is prime in $\C[x]$. We show that if a polynomial 
$f\in\R[x]$ is identically zero on $\cV(I)$,
then $f=0$ on $\cV(I')$; which implies $f\in I'\cap \R[x]=I$
i.e. $I$ is real.

Let $f_1,\dots,f_k$ generate $I$ and 
denote by $\Q'$ the field obtained by the extension of $\Q$ 
by the coefficients of $f,f_1,\dots,f_k$.
There exists a point $x$ in $\cV(I) \subset \R^n$ whose
transcendence degree is $\dim I$ on $\Q'$. Such a point
is a generic point of $\cV(I')$ on $\Q'$,
from Theorem~2 of Chapter IV of \cite{Weil62}.   
Thus $f$ is identically zero on $\cV(I')$.   
\end{proof}

\begin{example}
(i) The ideal $I=\<xy\>$ is decomposed as $I=\<x\>\cap \<y\>$.
$\dim^{{\rm top}}\cV(\<x\>)=\dim \<x\>=1$ implies $\<x\>$ is real prime,
and similarly $\<y\>$ is real prime. Hence $I$ is real.\\
(ii) The ideal $I=\<y^2-xz, x^3-yz\>$ is decomposed as $I=J\cap \<x,y\>$,
where $J=\<y^2-xz, x^3-yz,x^2y-z^2 \>$. For each,
$\dim^{{\rm top}} \cV(J)=\dim\{(t^3,t^4,t^5)~;~ t\in \R\}=\dim J=1$ implies 
$J$ is real prime,
and $\dim^{{\rm top}} \cV(\<x,y\>)=\dim \<x,y\>=1$ implies $\<x,y\>$
 is real prime. Hence $I$ is real.
\end{example}

Now we return to the semialgebraic set $S$.
We recall that $S^{\circ}$ is the interior of 
a semialgebraic set $S$ in $\R^n$.

\begin{theorem}
\label{condition}
Suppose that $S^{\circ} \cap\cV(I)$ is nonempty and
that $I=I_1\cap \cdots \cap I_k$ is the prime decomposition
of $I$. Then the following conditions are equivalent.\\
(i)\ $\cI(S^\circ \cap \cV(I))=I$;\\
(ii) For any $t$ ($1\leq t\leq k$), there exists
$x^t$ in $\cV(I_t)\cap S^\circ$ such that $\rank_{x^t} I_t= n-\dim I_t$;\\
(iii)  For any $t$ ($1\leq t\leq k$), $\dim^{{\rm top}} 
\cV(I_t)\cap S^\circ = \dim I_t$ holds.  
\end{theorem}

\begin{proof}
(i) $\Rightarrow$ (ii)\ 
Suppose that there exists $t$ ($1\leq t \leq k$) such that
$\rank_{x^t} I_t< n-\dim I_t$ for any $x^t \in \cV(I_t)\cap S^\circ$.
From Theorem~\ref{realcond} (i), there exists $x^t$ in $\cV(I_t)$
such that $\rank_{x^t} +\dim I_t =n$, and hence
the set $\{x\in \cV(I_t) ~;~ \rank_x I_t \leq n-\dim I_t-1\}$  
is a proper subvarietiy of 
$\cV(I_t)$. Hence, there exists a polynomial $f_t$ 
identically zero on $\cV(I_t)\cap S^\circ$ and 
not identically zero on $\cV(I_t)$.
Thus $f_t \not\in I_t$. Set $f\in\R[x]$ as $f=\prod_{s=1}^kf_s$, where
$f_s\in I_s \setminus I_t$ for $s\neq t$. Then $f$ is identically zero
on $\cV(I)\cap S^{\circ}$ and $f\not\in I$, which is a contradiction.\\
(ii)  $\Rightarrow$ (iii)\ 
Suppose that there exists
$x^t$ in $\cV(I_t)\cap S^\circ$ such that $\rank_{x^t} I_t= n-\dim I_t$.
From the proof of Theorem~\ref{realcond}, there exists a neighborhood
$U_{x^t}$ of $x^t$ in $\R^n$ such that
$\dim^{{\rm top}} \cV(I_t) \cap U_{x^t} = \dim I_t$.\\
(iii) $\Rightarrow$ (i)
It is clear that $\dim^{{\rm top}} \cV(I_t)\cap S^\circ = \dim I_t$ implies
$\dim^{{\rm top}} \cV(I_t) = \dim I_t$. The assertion follows from
Theorem~\ref{condtopdim} and Proposition~\ref{decomposition}.
\end{proof}

We give a proof of Theorem~\ref{main_condition}.

\begin{proof}
[Proof of Theorem~\ref{main_condition}]
(i) $\Rightarrow$ (ii)\  
Suppose that $\dim^{{\rm top}} \cV(I_t)\cap S < \dim I_t$ for some $t$.
From Theorem~\ref{condtopdim}, $\dim^{{\rm top}}\cV(I_t)= \dim I_t$. 
Thus, $\cV(I_t)\cap S$ is included in some proper subvarieties 
of $\cV(I_t)$. 
Hence, there exists a polynomial $f_t$ 
such that $f_t=0$ on $\cV(I_t)\cap S$ 
and $f_t$ is not identically zero on $\cV(I_t)$.
Thus $f_t\not\in I_t$, which yields a contradiction similarly to the proof
of Theorem~\ref{condition} (i) $\Rightarrow$ (ii).\\
(ii) $\Rightarrow$ (i)\
Replace $\cV(I_t)\cap S^\circ$ by $\cV(I_t)\cap S$ in the proof of 
Theorem~\ref{condition} (iii) $\Rightarrow$ (i).
\end{proof}

\begin{remark}
The condition obtained by replacing $S^\circ$ by 
$S$ in (ii) of Theorem~\ref{condition}:
``For any $t$ ($1\leq t\leq k$), there exists
$x^t$ in $\cV(I_t)\cap S$ such that $\rank_{x^t} I_t= n-\dim I_t$.''
does not guarantee (i) of Theorem~\ref{main_condition}. 
Indeed, set $S=\{(x,y) ~;~ 
g(x,y)=1-x^2-(y-1)^2 \geq 0\}$ 
and $h(x,y)=y$, then $I=\<y\>$, the origin O is in $\cV(I)$ and 
$\rank_{\rm O} I=1=2-\dim I$. However 
$\cI(\cV(I)\cap S)=\cI({\rm O})=\<x,y\>$ is not included in $I$.
\end{remark}


\section{Algorithms for testing or guaranteeing the duality}

In this section 
we propose an algorithm to calculate generators of ideal $\cI(K)
= \cI(S\cap \cV(I))$.
Applying it,
one can obtain an equivalent problem to POP (\ref{pop}) 
such that the resulting problem satisfies Condition (ii) of 
Theorem~\ref{main_condition}. 
The algorithm uses a part of the cylindrical algebraic 
decomposition (CAD) 
 after G. E. Collins (see \cite{BPR06,Mishra93} and 
references therein for basic literature).

The following algorithm is for detecting
whether the condition of $\cI(K)=I$ holds or not.
Note that if this condition holds, $I$ itself should be real,
because $I\subset \cI(\cV(I)) \subset \cI(K)=I$.
We omit details of the CAD procedures, which are illustrated in the
examples below.

Let $g_1,\dots,g_\ell$ and $h_1,\dots,h_m$ be 
defining polynomials of the semialgebraic set $S$ and 
of generators of the ideal $I$ respectively, i.e.
$\{g_1,\dots,g_\ell\}=\{g_{i,j} ~;~ 1\leq i\leq s,~ 1\leq j \leq r_i\} $
in \eqref{semialgebraic}.   

\begin{algorithm}\ Input: $S$ and $I$.
\begin{enumerate}
\item Compute the primary decomposition of $I$,
$I=I_1\cap \dots\cap I_k$. If $\C[x]I_t$ is not prime in $\C[x]$ 
for some $t$, then $I$ is not real, otherwise, go to (2).

\item For each $I_t=\<p_1,\dots,p_s\>$ do:
\begin{enumerate}
\item Choose coordinates $\tilde{x}=(x_{i_1},\dots,x_{i_{d_t}})$, 
where $d_t=\dim I_t$,
such that $1\not\in \C(\tilde{x})I_j$.
Here $\C(\tilde{x})$ denotes the field extended by 
$\{x_{i_1},\dots,x_{i_{d_t}}\}$ from $\C$.

\item Let $P^n$ and $Q^n$ denote the set of 
polynomials $\{p_1,\dots,p_s\}$
and $\{p_1,\dots,p_s,g_1,\dots,g_\ell\}$ respectively.
Execute the projection of CAD for the polynomial sets
$P^n$ and $Q^n$
from $\R^n$ to $\R^1$, where $\tilde{x}\in \R^{d_t}$.
For $d_t\leq n' \leq n$ let $P^{n'}$ and $Q^{n'}$ denote 
the set of irreducible factors of resulting polynomials on $\R^{n'}$ from $P^n$
and from $Q^{n}$ respectively.
Also let $C^{n'}$ denote the set of cells in $\R^{n'}$ from $Q^n$.

\item 
For any open cell $U_s \in C^{d_t}$,
take a sample point $\tilde{x}_s \in U_s\subset \R^{d_t}$. 

\item Lift $\tilde{x}_s$ to the point where some polynomial
$p \in P^{d_t+1}$ is zero, i.e. $p(\tilde{x}_s, x_{i_{d_t+1}})=0$. 
Denote $x_s^{d_t+1}=(\tilde{x}_s, x_{i_{d_t+1}})$
($x_s^{d_t+1}$ can be more than one point). 

\item Iterate the above step to the top level. 
Condition (ii) of Theorem~\ref{main_condition} holds if and only if
there exists a point $\tilde{x}_s$ such that the point can be lifted
to a point $x_s^n \in \R^n$ belonging to $S$.
\end{enumerate}
\end{enumerate}
\end{algorithm}

\begin{remark}
One can compute the primary decomposition, 
by using computer algebra systems; e.g.\ Macaulay2, Singular, Risa/Asir.
\end{remark}
\begin{remark}
The condition $1\not\in \C(\tilde{x})I$ is equivalent to the condition that 
the topological dimension of projection of $\cV(I')\subset \C^n$ to 
$\C^d$ is $d$.
\end{remark}

\begin{example}
Set $h_1=y^2-xz$, $h_2=x^3-yz$, $h_3=x^2y-z^2$ and 
$g_1=1-(x-1)^2-(y-1)^2-(z-1)^2$. 
Applying Algorithm~4.1, we have the following:\\
(1) $I=\<h_1,h_2,h_3\>$ is prime and $\C[x]I$ is also prime.\\
(2) The dimension of $I$ is 1 and $1\not\in \C(x_1)I$.
The set of irreducible factors of the resultants and sub-resultants for $I$ 
with respect to $z$ is $P^2=\{x, y, x^4-y^3\}$.
That for $I$ and $g_1$ is 
\begin{align*}
Q^2=\{&x, y, x^4-y^3, 1-(x-1)^2-(1-y)^2,\\
&x^2(x-1)^2+x^2(y-1)^2+y^4-2xy^2,
y^2(x-1)^2+y^2(y-1)^2+x^6-2x^3y,\\
&\left|\begin{array}{cccc}
-1&0&x^2y&0\\
0&-1&0&x^2y\\
1&-2&(x-1)^+(y-1)^2&0\\
0&1&-2&(x-1)^2+(y-1)^2
\end{array}\right|\}
\end{align*}
(Fig.~\ref{ex4.1}).
The irreducible factors of the resultants and sub-resultants for $Q^2$
with respect to $y$ is approximately
$$\{x,x-2,x-0.522613, x-1.39169, x-0.714577, x-1.74196, {\rm etc}\}.$$
Let $x=0.6^{\frac{3}{4}}$ as a sample point of the interval 
$(0.5226, 0,714577) \in C^1$.
By lifting it to $xy$-plane by $P^2$, we obtain two points
$(0.6^{\frac{3}{4}},0), (0.6^{\frac{3}{4}},0.6^{\frac{4}{3}})$.
Further, by lifting it to $xyz$-plane by $I$, we obtain only one point
$(0.6^{\frac{3}{4}},0.6^{\frac{4}{3}}, 0.6^{\frac{5}{3}})$,
which satisfies $g_1\geq 0$. Thus, $\cI(K)=I$ holds.
\end{example}

\begin{figure}[htbp]
\begin{center}
\begin{picture}(200,200)(0,0)
\includegraphics*[scale=0.5]{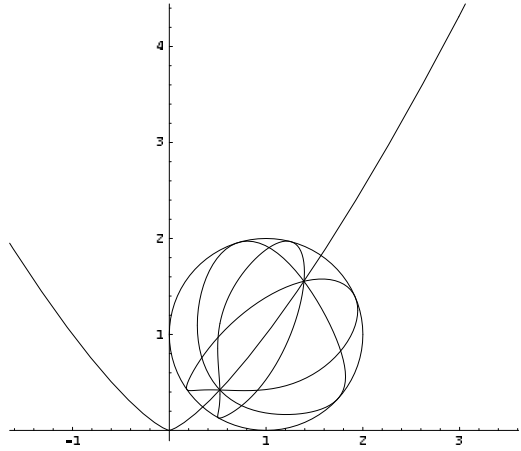}
\end{picture}
\caption{Real roots of each polynomial in $Q^2$ in Example~\ref{ex_4.1};
 The results of projection to $xy$-plane by the CAD}\label{ex4.1}
\end{center}
\end{figure}

Even if $\cI(K)\neq I$, we can generate new equality constraints for POP (\ref{pop}) by the following algorithm.

\begin{algorithm}\
Input: $S$ and $I$ such that $\cI(K)\neq I$. 
\begin{enumerate}
\item 
\begin{enumerate}
\item
Compute the primary decomposition of $I$,
$I=I_1\cap \dots\cap I_k$. 
\item
If $I_t$ is not prime in $\R[x]$,
replace $I_t$ by its associated prime $\sqrt{I_t}$. 
\item
If $I_t'=\C[x]I_t$ is not prime in $\C[x]$, 
$I_t'$ is decomposed as Proposition~\ref{Ig}.
Add $\Rep g_i , \Imp g_i$ to $I$ and 
go back to (a).
\end{enumerate}
\item If necessary, operate an invertible linear transformation on $\R^n$
so that $1\not\in \C(x_1,\dots,x_{d_t}) I_t$ holds for all $t$, 
where $d_t$ denotes $\dim I_t$.  

\item Let $t_0$ be one of the numbers satisfying 
$\cI(\cV(I_{t_0})\cap S)\neq I_{t_0}$ and
$\dim I_{t_0}=d_{\max}$, 
where $d_{\max} =\max \{\dim I_t ~;~ 1\leq t\leq k,
\cI(\cV(I_{t})\cap S)\neq I_t\}$.
Such $t_0$ can be calculated by Algorithm~4.1.

\item
For each $I_t$, if $\cI(\cV(I_t)\cap S)= I_t$ let $f^t$ be a polynomial
in $I_t\setminus I_{t_0}$, otherwise do for $I_t=\<p_1,\dots,p_s\>$:
\begin{enumerate}
\item Let $P^n$ and $Q^n$ denote the set of 
polynomials $\{p_1,\dots,p_s\}$
and $\{p_1,\dots,p_s,g_1,\dots,g_\ell\}$ respectively.
Execute the projection of CAD for the polynomial sets
$P^n$ and $Q^n$
from $\R^n$ to $\R^1$, eliminating  $x_n,\dots,x_2$ in this ordering.
For $d\leq n' \leq n$, let $P^{n'}$ and $Q^{n'}$ denote the set of 
irreducible factors of
resulting polynomials on $\R^{n'}$ from $P^n$ and from $Q^{n}$.
Also let $C^{n'}$ denote the set of cells in $\R^{n'}$ from $Q^n$.

\item 
Let $\{V_s\}$ denotes the set of cells in $\C^n$ 
such that $V_s$ is included in $\cV(I_t)\cap S$. 
By assumption $\cI(\cV(I_{t})\cap S)\neq I_t$, 
the dimension of $\Proj_{d_t} V_s$ (the projection of $V_s$ to $\R^{d_t}$) 
is less than $d_t$.
Define $f_s^t$ as one of the irreducible defining polynomials of 
$\Proj_{d_t} V_s$ in $Q^{d_t}$, and let $f^t$ be 
the least common multiple of $f_s^t$'s for all $s$. 
Since $1 \not\in \C(x_1,\dots,x_{d_t})I_{t'}$ for all 
$t'$ such that $d_{t'} \geq d_t$, while 
$1\in \C(x_1,\dots,x_{d_t})\<f^t(x_1,\dots,x_{d_t})\>$,
the polynomial $f^t$ is not included in $I_{t'}$
for all $t'$ such that $d_{t'} \geq d_t$.  
\end{enumerate}
\item Replace $I$ by $\<I_1,f^1\>\cap \dots \cap \<I_k,f^k\>$
($g_1,\dots,g_\ell$ are not changed)
and return to Algorithm~4.1.
\end{enumerate}
\end{algorithm}

By the following proposition, the resulting ideal 
is strictly larger than $I$. By Noether property, these procedures
terminate in finite steps.

\begin{proposition}
If $\cI(K)\neq I$, then
$\<I_1,f^1\>\cap \dots \cap \<I_k,f^k\>$ in Algorithm~4.2 
is strictly larger than $I$. 
\end{proposition}

\begin{proof}
We show $f=\prod_t f^t \not\in I_{t_0}$, 
which implies $f \not\in I=I_1 \cap \dots \cap I_k$.
Suppose $f \in I_{t_0}$, then since $I_{t_0}$ is prime, $f^t$
belongs to $I_{t_0}$ for some $t$. 
If $\cI(\cV(I_t)\cap S)=I_t$, then by the beginning of (4), 
$f^t$ does
not belong to $I_{t_0}$. Thus, we have $\cI(\cV(I_t)\cap S)\neq I_t$.
Hence by (4.b), $f^t$ does not belong to $I_{t'}$ for all $t'$ satisfying 
$d_{t'} \geq d_t$, while we have supposed $f^t\in I_{t_0}$. 
Thus we have $d_{t_0} < d_t$, which contradicts 
$\cI(\cV(I_t)\cap S)\neq I_t$ by the definition of $t_0$ in (3). 
\end{proof}

\begin{example}
\label{ex_4.1}
Set $h_1=(x^2+y^2+z^2)(z-2)$ and 
$g_1=1-x^2-(z-1)^2$. Applying Algorithm~4.2, we have the following.\\
(1) $I=I_1\cap I_2$, where $I_1=\<x^2+y^2+z^2\>$ and $I_2=\<z-2\>$, 
is the prime decomposition and $\C[x]I_1$ and $\C[x]I_2$ are prime.\\
(2) Both dimensions of $I_1$ and $I_2$ are 2 and $1\not\in \C(x,y)I$.
The set of irreducible factors of the resultants and sub-resultants for $I_1$
and $g_1$ with respect to $z$ is
$Q_1^2=\{x+1,x-1,x^2+y^2,4x^2+4y^2+y^4\}$.
The irreducible factors of the resultants and sub-resultants for $Q_1^2$
with respect to $y$ is
$\{x,x-1,x+1\}$, thus we have 
$$C_1^1=\{(-\infty,-1),-1,(-1,0),0,(0,1),1,(1,\infty)\}.$$
By lifting $C_1^1$  to $xy$-plane by $Q_1^2$, we obtain
\begin{align*}
C_1^2=\{& \{x<-1 \},\{x=-1\},\{-1< x< 0\},
\{x=0, y<0\}, \{x=0, y=0\} \\
&\{x=0, y>0\}, \{0<x<1\},\{x=1\},
\{x>1\} \}.
\end{align*}
Further, by lifting it to $xyz$-plane by $P^3=I_1$, we obtain only one point
$U_1^2=(0,0,0)$ satisfying $g_1\geq 0$. 
We take defining polynomial $f^1$ of $(0,0)$ from $Q^2$,
e.g. $f^1=x^2+y^2$.

Similarly, from $I_2$ and $g_1$ we obtain
$Q_2^1=Q_2^2=\{x,x+1,x-1\}$. By lifting them to $xyz$-plane by $P^3=I_2$, 
we obtain $\{x<0, z=2\}, \{x=0, z=2\},\{x>0, z=2\}$,
among which only $U_2^2=\{x=0, z=2\}$ satisfy $g_1\geq 0$.
We take defining polynomial $f^2$ of $\{x=0\} \subset \R^2$ from $Q^2$,
i.e. $f^2=x$.

We replace $I$ by 
$$I=I_1\cap I_2=\<x^2+y^2+z^2,x^2+y^2\> \cap \<z-2,x\>
=
\<x^2+y^2,z^2\> \cap \<x,z-2\>
$$
(this does not satisfy $\cI(K)=I$ yet).
From $x^2+y^2=(x+y\I)(x-y\I)$, we should add
$x,y$ to $I_1$. Thus we should replace $I$ by
$$I=
\<x,y,z\> \cap \<x,z-2\>=
\<x,y(z-2),z(z-2)\>.$$
Now, if we set
$h_1=x,h_2=y(z-2),h_3=z(z-2)$ and $g_1=1-x^2-(z-1)^2\geq 0$,
then $\cI(K)=I$ is satisfied. 
\end{example}

\subsection*{Acknowledgement}
Y.~T and T.~T. appreciate the financial support from the Faculty of Marin 
Technology, Tokyo University of Marine Sciences and Technology.
They are also supported by the Japan 
Society for the Promotion of Science as Grand-in-Aid for Young Scientists (B). 
H.~ W was supported by Grant-in-Aid for JSPS Fellows 20003236.

\end{document}